\documentclass{amsart}

\def\dv{\mathop{\rm div}}

\def\dx{\hspace{2pt}{\rm d}x}

\def\l2{_{L_2(\Omega)}}
\def\nl2{_{[L_2(\Omega)]^n}}
\def\T{\mathcal{T}}
\def\d{\hspace{2pt}{\rm d}}

\newcounter{saveeqn}

\newtheorem{theorem}{Theorem}[section]
\newtheorem{lemma}[theorem]{Lemma}

\newtheorem{remark}[theorem]{\it Remark}

\theoremstyle{definition}

\theoremstyle{remark}

\numberwithin{equation}{section}

\begin{document}

\title{Local energy estimates for the finite element method on sharply varying grids}

\author{Alan Demlow}
\address{University of Kentucky \\ Department of Mathematics \\ 715 Patterson Office Tower \\ Lexington, KY  40506--0027}
\email{demlow@ms.uky.edu}
\thanks{The first author was partially supported by NSF grant DMS-0713770.}

\author{Johnny Guzm\'an}
\address{Brown University \\  Division of Applied Mathematics \\ 182 George St.\\ Providence, RI  02906}
\email{johnny\_guzman@brown.edu}
\thanks{The second author was partially supported by NSF grant DMS-0503050.}

\author{Alfred H. Schatz}
\address{Department of Mathematics, Malott Hall\\Cornell University\\Ithaca, NY  14853}
\thanks{The third author was partially supported by NSF grant DMS-0612599.  }




\begin{abstract}  Local energy error estimates for the finite element method for elliptic problems were originally proved in 1974 by Nitsche and Schatz.  These estimates show that the local energy error may be bounded by a local approximation term, plus a global ``pollution'' term that measures the influence of solution quality from outside the domain of interest and is heuristically of higher order.  However, the original analysis of Nitsche and Schatz is restricted to quasi-uniform grids.  We present local a priori energy estimates that are valid on shape regular grids, an assumption which allows for highly graded meshes and which much more closely matches the typical practical situation.  Our chief technical innovation is an improved superapproximation result. 
\end{abstract}
 
\maketitle

\markboth{A. DEMLOW, J. G\'UZMAN, AND A. SCHATZ}{LOCAL ENERGY ESTIMATES ON SHARPLY VARYING GRIDS}

\section{Introduction}
In this note we prove local energy error estimates for the finite element method for second-order linear elliptic problems on highly refined triangulations.  Most a priori error analyses for the finite element method in norms other than the global energy norm place severe restrictions on the mesh.   In particular, such error analyses are most often carried out under the assumption that the grid is {\it quasi uniform}, that is, all simplices in the mesh are required to have diameter equivalent to some fixed parameter $h$.  The typical practical situation is rather different.  Many (especially adaptive) finite element codes enforce only {\it shape regularity} of elements, meaning that all elements in the mesh must have bounded aspect ratio.  Though it places a weak restriction upon the rate with which the diameters of elements in the mesh may change, shape regularity allows for the locally refined meshes that are needed to resolve the singularities and other sharp local variations of the solution that occur in the majority of practical applications.

In the work \cite{NS74} of Nitsche and Schatz, local energy error estimates were established for  interior subdomains under the assumption that the finite element grid is quasi-uniform.  Such local energy estimates are helpful in understanding basic error behavior, especially ``pollution effects'' of global solution properties on local approximation quality, and they also provide an important technical tool in many proofs of pointwise bounds for the finite element method (cf. \cite{SW95}).  In addition, the most relevant error notion in applications is often related to some {\it local} norm or functional instead of to the global energy error, as evidenced by the recent surge of interest in ensuring control of the error in calculating  ``quantities of interest'' in adaptive finite element calculations instead of merely controlling the default global energy error (cf. \cite{BR01}).    As a final example of the applicability of local energy estimates, we mention that the estimates of \cite{NS74} have been used to justify certain approaches to parallelization and adaptive meshing (cf. \cite{BH00}).  Thus local energy estimates are of broad and fundamental importance in finite element theory.  

Here we prove local energy error estimates under the assumption that the finite element triangulation is shape regular instead of under the more restrictive assumption of quasi uniformity required in \cite{NS74}.  In other words, we essentially prove that the results of Nitsche and Schatz hold under the restrictions typically placed upon meshes in practical codes, which in particular allow for highly graded grids.  Our main innovation is a novel ``superapproximation'' result which we state and prove in \S2.  In \S3 we then prove a local energy bound that is valid on grids that are only assumed to be shape-regular.  As in \cite{NS74}, our results are valid for operators that are only {\it locally} elliptic, so that the PDE under consideration may be degenerate or change type outside of the domain of interest.  In contrast to \cite{NS74}, the results we present here are valid up to the domain boundary, allow for nonhomogeneous Neumann, Dirichlet, and mixed boundary conditions, and also require only $L_\infty$ regularity of the coefficients of the differential operator.  

\section{An improved superapproximation result}
\label{sec2}
An essential feature of the proofs of local error estimates given in \cite{NS74}, and also of essentially all published proofs of local and maximum-norm a priori error estimates for finite element methods, is the use of superapproximation properties.  In essence, superapproximation bounds establish that a function in the finite element space multiplied by any smooth function can be approximated exceptionally well by the finite element space. 

In order to fix thoughts, we shall in this section assume for simplicity that $\Omega \subset \mathbb{R}^n$ is a polyhedral domain; a more general situation is considered in \S3 below.  Let $\T_h$ be a simplicial decomposition of $\Omega$.  Denote by $h_T$ the diameter of the element $T \in \T_h$.  We assume throughout that the elements in $\T_h$ are shape-regular, that is, each simplex $T \in \T_h$ contains a ball of diameter $c_1 h_T$ and is contained in a ball of radius $C_1 h_T$, where $c_1$ and $C_1$ are fixed.  Let also $S_h^r$ be a standard Lagrange finite element space consisting of continuous piecewise polynomials of degree $r-1$.   We shall use standard notation for Sobolev spaces, norms, and seminorms, e.g., $\|u\|_{H^1(\Omega)}=(\int_\Omega (u^2+|\nabla u|^2) \dx )^{1/2}$, $|u|_{W_p^k(\Omega)}=(\sum_{|\alpha|=k} \|D^\alpha u\|_{L_p(\Omega)}^p)^{1/p}$, etc.  

A standard superapproximation result is as follows.  Let $\omega \in C^\infty(\Omega)$ with $|\omega|_{W_\infty^j(\Omega)} \le C d^{-j}$, $0 \le j \le r$.  Then for each $\chi \in S_h^r$, there exists $\eta \in S_h^r$ such that for each $T \in \T_h$ satisfying $d \ge h_T$, 
\begin{equation}
\|\omega \chi-\eta\|_{H^1(T)} \le C( \frac{h_T}{d} \|\nabla \chi\|_{L_2(T)}+ \frac{h_T}{d^2} \|\chi\|_{L_2(T)}).
\label{eq2-1}
\end{equation}
Our modified result follows (cf. \cite{Guz06}).
\begin{theorem}
\label{t2-1}
Let $\omega \in C^\infty(\Omega)$ with $|\omega|_{W_\infty^j(\Omega)} \le C d^{-j}$ for $0 \le j \le r$.  Then for each $\chi \in S_h^r$, there exists $\eta \in S_h^r$ such that for each $T \in \T_h$ satisfying $d \ge h_T$,
\begin{equation}
\| \omega^2 \chi - \eta\|_{H^1(T)} \le C (\frac{h_T}{d} \|\nabla (\omega \chi)\|_{L_2(T)}+ \frac{h_T}{d^2} \|\chi\|_{L_2(T)}).  
\label{eq2-2}
\end{equation}
\end{theorem}

\begin{remark} {\rm
There are two differences between (\ref{eq2-1}) and (\ref{eq2-2}).  First, in (\ref{eq2-1}) we consider approximation of $\omega \chi$, whereas in (\ref{eq2-2}) we consider approximation of $\omega^2 \chi$.  Secondly, in (\ref{eq2-1}) the norms on the right hand side involve only $\chi$, whereas in (\ref{eq2-2}) the $H^1$ seminorm involves $\omega \chi$.  If we think of $\omega$ as a cutoff function, this distinction becomes vitally important:  $\omega \chi$ has the same support as $\omega^2 \chi$, whereas the support of $\chi$ is generally larger than that of $\omega \chi$.  This seemingly minor difference will allow us to establish local energy estimates on grids that are only assumed to be shape regular.
}\end{remark}

\begin{proof} Let $I_h:C^0(\Omega) \rightarrow S_h^r$ be the standard Lagrange interpolant.  We shall choose $\eta=I_h (\omega^2 \chi)$ in (\ref{eq2-2}).  For $T \in \T_h$, we may use standard approximation theory (cf. \cite{BS02}) to calculate
\begin{equation}
\begin{split}
\|\omega^2 \chi-I_h(\omega^2\chi)\|_{H^1(T)} \le & C h_T^{n/2} \|\omega^2 \chi-I_h (\omega^2 \chi)\|_{W_\infty^1(T)}
\\ \le & C h_T^{n/2+r-1} | \omega^2 \chi|_{W_\infty^{r}(T)}.
\end{split}
\label{eq2-3}
\end{equation}
Noting that $D^\alpha \chi=0$ for all multiindices $\alpha$ with $|\alpha|=r$, recalling that $\frac{h_T}{d} \le 1$, and employing inverse estimates, we compute
\begin{equation}
\begin{split}
C h_T^{n/2+r-1} & | \omega^2 \chi|_{W_\infty^{r}(T)} \le  C (\sum_{i=2}^r h_T^{i-1} |\omega^2|_{W_\infty^i(T)}) \|\chi\|_{L_2(T)}
\\ & + Ch_T^{n/2+r-1} \sum_{|\alpha|=1, |\beta|=r-1} \|D^\alpha \omega^2 D^\beta \chi\|_{L_\infty(T)}
\\ \le & C \frac{h_T}{d^2} \|\chi\|_{L_2(T)}+Ch_T^{n/2+r-1} \sum_{|\alpha|=1, |\beta|=r-1} \|D^\alpha \omega^2 D^\beta \chi\|_{L_\infty(T)}.  
\end{split}
\label{eq2-4} 
\end{equation}

We next consider the terms $\|D^\alpha \omega^2 D^\beta \chi\|_{L_\infty(T)}$ above.  Since $|\alpha|=1$, we have $D^\alpha \omega^2=2 \omega D^\alpha \omega$.  Let $\hat{\omega}=\frac{1}{|T|} \int_T \omega \dx$ so that $\|\omega-\hat{\omega}\|_{L_\infty(T)} \le C h_T |\omega|_{W_\infty^1(T)} \le C \frac{h_T}{d}$.  Employing inverse estimates, we thus have
\begin{equation}
\begin{split}
Ch_T^{n/2+r-1} & \sum_{|\alpha|=1, |\beta|=r-1} \|D^\alpha \omega^2 D^\beta \chi\|_{L_\infty(T)} 
\\ \le & C d^{-1} h_T^{n/2+r-1} \sum_{|\beta|=r-1} \|\omega D^\beta \chi\|_{L_\infty(T)}
\\  \le & C d^{-1} h_T^{n/2+r-1} \sum_{|\beta|=r-1} (\|(\omega-\hat{\omega}) D^\beta \chi\|_{L_\infty(T)}+\|\hat{\omega} D^\beta \chi\|_{L_\infty(T)})
\\ \le & C (\frac{h_T}{d^2} \|\chi\|_{L_2(T)}+\frac{h_T}{d} |\hat{\omega} \chi|_{H^1(T)})
\\ \le & C (\frac{h_T}{d^2} \|\chi\|_{L_2(T)}+\frac{h_T}{d} |(\hat{\omega}-\omega) \chi|_{H^1(T)}+\frac{h_T}{d} |\omega \chi|_{H^1(T)}).
\end{split}
\label{eq2-5}
\end{equation}
Using an inverse inequality, we find that
\begin{equation}
\begin{split}
\frac{h_T}{d} |(\hat{\omega}-\omega) \chi|_{H^1(T)} \le &\frac{h_T}{d} (|\omega|_{W_\infty^1(T)} \|\chi\|_{L_2(T)}+ \|\hat{\omega}-\omega\|_{L_\infty(T)} |\chi|_{H^1(T)})
\\ \le & C\frac{h_T}{d} (\frac{1}{d} \|\chi\|_{L_2(T)}+\frac{h_T}{d} |\chi|_{H^1(T)})
\\ \le & C\frac{h_T}{d^2} \|\chi\|_{L_2(T)}.
\end{split}
\label{eq2-6}
\end{equation}
Inserting (\ref{eq2-6}) into (\ref{eq2-5}) and the result into (\ref{eq2-4}) and (\ref{eq2-3}) completes the proof of (\ref{eq2-2}). 
\end{proof}

\section{Local $H^1$ estimates}
In this section we state and prove a local $H^1$ estimate that is valid on highly graded grids.  We now let $\Omega$ be a domain in $\mathbb{R}^n$, and let $\Omega_0$ be a bounded subdomain of $\Omega$.  We decompose $\partial \Omega \cap \partial \Omega_0$ (if it is nonempty) into a Dirichlet portion $\Gamma_D$ and a Neumann portion $\Gamma_N$.  For the sake of simplicity, we assume that $\Gamma_D$ is polyhedral and that $\Gamma_N$ is either polyhedral or Lipschitz.  Let $u$ satisfy
\begin{equation}
\begin{split}
-\dv(A \nabla u)+b\cdot \nabla u+cu=&f \hbox{ in } \Omega_0,
\\u=&g_D \hbox{ on } \Gamma_D,
\\ \frac{\partial u}{\partial n_A}=&g_N \hbox{ on } \Gamma_N.
\end{split}
\label{eq3-1}
\end{equation}
Here $A$ is an $n \times n$ coefficient matrix that is uniformly bounded and positive definite in $\Omega$, $b \in L_\infty(\Omega_0)^n$, $c \in L_\infty(\Omega_0)$, and $\frac{\partial}{\partial n_A}$ is the conormal derivative with respect to $A$.  We also assume that $\Omega \subset \mathbb{R}^n$.   Note that we make no assumptions about the differential equation solved by $u$ outside of $\Omega_0$.

Let $H_{D,0}^1(\Omega_0)=\{ u \in H^1(\Omega_0):u|_{\Gamma_D}=0 \}$, and let $H_{D}^1(\Omega_0)=u \in H^1(\Omega_0):u|_{\Gamma_D}=g_D\}$.  Also let  $H_<^1(B)=\{u \in H^1(\Omega_0): u|_{\Omega \setminus B}=0\}$ for subsets $B$ of $\Omega_0$.  Thus functions in $H_<^1(B)$ are zero on $\partial B \setminus \partial \Omega$, but may be nonzero on portions of $\partial B$ coinciding with $\partial \Omega$, or put in other terms, functions in $H_<^1(B)$ are compactly supported in $B$ modulo $\partial \Omega$.  Rewriting (\ref{eq3-1}) in its weak form, we find that $u \in H_{D}^1(\Omega_0)$ satisfies
\begin{equation}
\begin{split} 
L(u,v):=&\int_\Omega (A \nabla u \nabla v +b\cdot \nabla u v+cuv)\dx
\\ =&\int_\Omega fv \dx-\int_{\Gamma_N} g_N v \d \sigma, ~ v \in H_{D,0}^1(\Omega) \cap H_<^1(\Omega_0).
\end{split}
\label{eq3-2}
\end{equation}

Following \cite{NS74}, we do not assume that $L$ is coercive over $H^1(\Omega_0)$, but rather we make a {\it local} coercivity assumption: 
\vspace{3pt}
\newline {\it R1:  Local coercivity.}  There exists a constant $d_0>0$ such that if $B$ is the intersection of any open sphere of diameter $d \le d_0$ with $\Omega_0$, then $L$ is coercive over $H_<^1(B)$, that is, for some constant $C_1>0$,
\begin{equation}
(C_1)^{-1} \|u\|_{H^1(B)}^2 \le L(u,u) \le C_1 \|u\|_{H^1(B)}^2, ~ u \in H_<^1(B).
\label{eq3-2-1}
\end{equation}

\begin{remark}
\label{rem3-1}
{\rm R1 may be satisfied in one of two ways.  It may happen that $L$ is coercive over $H^1(\Omega_0)$, in which case no further argument is needed.  R1 so long as a Poincar\'e inequality 
\begin{equation}
\|u\|_{L_2 (B)} \le Cd\|u\|_{H^1(B)}
\label{eq3-2-1a}
\end{equation}
holds for balls $B$ as in R1 having small enough diameter (cf. Remark 1.2 of \cite{NS74}).  Such Poincar\'e inequalities always hold for interior balls.  If $B$ is the nontrivial intersection of an open ball with $\Omega$, then (\ref{eq3-2-1a}) holds for $d \le d_1$ small enough under the restrictions we have placed on $\partial \Omega \cap \partial \Omega_0$; here $d_1$ depends on the properties of $\partial \Omega \cap \partial \Omega_0$. } \end{remark}

Next we make assumptions concerning the finite element approximation $u_h$ of $u$.  Let $\T_0$ be a triangulation such that $\Omega_0 \subset \cup_{T \in \T_0} \overline{T}$ and $T \cap \Omega_0 \neq \emptyset$ for all $T \in \T_0$.  Let $h_T=diam(T)$ for $T \in \T_0$.  We denote our trial finite element space by $S_D$.  We do not assume that $S_D \subset H_D^1(\Omega)$.  In addition, we let $S_{D,0}=S_D \cap  H_{D,0}^1(\Omega_0)$ be our trial finite element space.  We assume that $u_h$ is the local finite element approximation to $u$ on $\Omega_0$, that is, $u_h \in S_D$ and 
\begin{equation}
L(u-u_h, v_h)=0 \hbox{ for all } v_h \in S_{D,0} \cap H_<^1(\Omega_0).
\label{eq3-2-3}
\end{equation}
We do not explicitly fix $u_h$ on the Dirichlet portion of the boundary, but rather implicitly assume that  $u_h|_{\Gamma_D}$ is set equal to some appropriate interpolant or projection of $g_D$.  

Next we state properties that $S_D$ and $S_{D,0}$ must possess in order to prove the desired local energy error estimate.  Let $\tilde{d} \le d_0$ be a fixed parameter, and let $G_1$ and $G$ be arbitrary subsets of $\Omega_0$ with $G_1 \subset G$ and $dist(G_1, \partial G \setminus \partial \Omega) =\tilde{d}>0$.  Then the following are assumed to hold:
\vspace{3pt}
\newline {\it A1:  Local interpolant.}  There exists a local interpolant $I$ such that for each $u \in H_<^1(G_1)$, $I u \in S_D  \cap H_<^1(G)$, and for each $ u \in H_{D,0}^1(\Omega_0)$, $I u \in S_{D,0}$.  
\vspace{3pt}
\newline {\it A2:  Inverse properties.}  For each $\chi \in S_D$, $T \in \T_h$, $1 \le p \le q \le \infty$, and $0 \le \nu \le s \le r$ with $r$ sufficiently small,
\begin{equation}
\|\chi\|_{W_q^s(T)} \le C h_T^{\nu-s +\frac{n}{p}-\frac{n}{q}} \|\chi\|_{W_p^\nu(T)}.
\label{eq3-2-4}
\end{equation}
\vspace{3pt}
\newline {\it A3:  Superapproximation.}  Let $\omega \in C^\infty(\Omega_0) \cap H_<^1(G_1)$ with $|\omega|_{W_\infty^j(\Omega_0)} \le C d^{-j}$ for integers $0 \le j \le r$ with $r$ sufficiently large.  For each $\chi \in S_{D,0}$  and for each $T \in \T_h$ satisfying $d \le h_T$,
\begin{equation}
\| \omega^2 \chi - I (\omega^2 \chi) \|_{H^1(T)}  \le C (\frac{h_T}{d} \|\nabla (\omega \chi) \|_{L_2(T)}+ \frac{h_T}{d^2} \| \chi\|_{L_2(T)}),
\label{eq3-2-2}
\end{equation}
where the interpolant $I$ is as in A1 above.  

\begin{remark}{\rm A1, A2, and A3 are satisfied by standard finite element spaces defined on shape-regular triangular grids.  A1 also essentially requires that the finite element mesh resolve $G \setminus G_1$, i.e., that $\tilde{d} \ge K \max_{T \cap G \neq \emptyset} h_T$ with $K$ large enough.  
} \end{remark}

We begin by proving a Caccioppoli-type estimate for ``discrete harmonic'' functions.  Such a statement was also proved in \cite{NS74} as a preliminary to local energy estimates, though the proof we give below more closely follows \cite{SW77}.

\begin{lemma}
\label{lem3-1}
Let $G_0 \subset G \subset \Omega_0$ be given, and let $dist(G_0, \partial G \setminus \partial \Omega)=d$ with $d \le 2 d_0$ where $d_0$ is the parameter defined in the assumption R1.  Let also A1, A2, and A3 hold with $\tilde{d}=\frac{d}{4}$, and assume that $u_h \in S_{D,0}$ satisfies
\begin{equation}
L(u_h, v_h)=0 \hbox{ for all } v_h \in S_{D,0} \cap H_<^1(\Omega_0).
\label{eq3-9-a}
\end{equation}
In addition let $\max_{T \cap G \neq \emptyset} \frac{h_T}{d} \le \frac{1}{4}$.  Then
\begin{equation}
\|u_h\|_{H^1(G_0)} \le 
 C\frac{1}{d} \|u_h\|_{L_2(G)}.
\label{eq3-9-b}
\end{equation}
Here $C$ depends only on the constants in (\ref{eq3-2-4}) and (\ref{eq3-2-2}) and the coefficients of $L$.
\end{lemma}

\begin{proof} We assume that $G_0$ is the intersection of a ball $B_{\frac{d}{4}}$ of radius $\frac{d}{4}$ with $\Omega_0$; the general case may be proved using a covering argument.  Let then $G_1$ and $G_2$ be the intersections with $\Omega_0$ of balls having the same center as $G_0$ and having radii $\frac{d}{2}$ and $\frac{3d}{4}$, respectively, and without loss of generality let $G$ be the corresponding ball of radius $d$.  Let then $\omega \in C_0^\infty(G_1)$ be a cutoff function which is $1$ on $G_0$ and which satisfies $\|\omega\|_{W_\infty^j(G_1)} \le Cd^{-j}$, $0 \le j \le r$.  We may then apply the assumptions A1 through A3 to the pairs $G_1$ and $G_2$, and $G_2$ and $G$.  

Using (\ref{eq3-2-1}), we first compute that 
\begin{equation}
\|u_h\|_{H^1(G_0)}^2 \le \|\omega u_h \|_{H^1(G)}^2 \le  CL(\omega u_h ,\omega u_h).
\label{eq3-4}
\end{equation}
Using the fact that $\|\nabla \omega \|_{L_\infty(\Omega)} \le \frac{C}{d}$, we compute that for any $\epsilon>0$,
\begin{equation}
\begin{split}
L(\omega&u_h, \omega u_h)= L(u_h,\omega^2 u_h)
 \\ & - \int_\Omega u_h [ A \nabla(\omega u_h) \nabla \omega +u_h A \nabla \omega \nabla \omega  
 + A \nabla \omega \nabla (\omega u_h) +\omega u_h b \nabla \omega ]\dx  
 \\ \le & |L(u_h, \omega^2 u_h)|+ C\frac{1}{d^2 \epsilon} \|u_h\|_{L_2(G)}^2+\epsilon \|\omega u_h \|_{H_1(G)}^2 .
\end{split}
\label{eq3-5}
\end{equation}
Next we use (\ref{eq3-9-a}), (\ref{eq3-2-2}),  and the fact that $\|\omega^2 u_h\|_{H^1(G)} \le \|\omega u_h\|_{H^1(G)}+\frac{C}{d}\|u_h\|_{L_2(G)}$ to compute
\begin{equation}
\begin{split}
L( u_h, \omega^2 u_h) = & L(u_h, \omega^2 u_h-I (\omega^2 u_h))
\\ \le & 
C \sum_{T \cap G_2 \neq \emptyset} h_T \|u_h\|_{H^1(T)} (\frac{1}{d} |\omega  u_h|_{H^1(T)}+\frac{1}{d^2} \|u_h\|_{L_2(T)}).
\end{split}
\label{eq3-6}
\end{equation}
Using (\ref{eq3-2-4}) and the fact that $\frac{h_T}{d} \le 1$, we have for $\epsilon$ as above that
\begin{equation}
\begin{split}
C h_T \|u_h&\|_{H^1(T)}   (\frac{1}{d} |\omega u_h|_{H^1(T)}+\frac{1}{d^2} \|u_h\|_{L_2(T)})
\\  \le & \frac{C}{\epsilon d^2} \|u_h\|_{L_2(T)}^2+\epsilon |\omega u_h|_{H^1(T)}^2.
\end{split}
\label{eq3-7}
\end{equation}

Inserting (\ref{eq3-7}) into (\ref{eq3-6}), noting that $T \cap G_2 \neq \emptyset$ implies that $T \subset G$ (since $\max_{T \cap G \neq \emptyset} h_T \le \frac{d}{4}$) and carrying out further elementary manipulations then yields that for $\epsilon>0$, 
\begin{equation}
L(u_h,  \omega^2 u_h ) 
\\ \le   \frac{C}{\epsilon d^2}\|u_h\|_{L_2(G)}^2+ \epsilon \|\omega u_h \|_{H^1(G)}^2.  
\label{eq3-8}
\end{equation}
Inserting (\ref{eq3-8} into (\ref{eq3-5}) and the result into (\ref{eq3-4}) yields
\begin{equation}
\|\omega u_h \|_{H^1(G)}^2 
 \le  \frac{C}{\epsilon d^2} \|u_h\|_{L_2(G)}^2+2 \epsilon \|\omega u_h\|_{H^1(G)}^2.
\label{eq3-9}
\end{equation}
Taking $\epsilon=\frac{1}{4}$ so that we may kick back the last term above, employing the triangle inequality, and inserting the result into (\ref{eq3-4}) then completes the proof of (\ref{eq3-9-b}).  
\end{proof}

We now prove a local energy error estimate.  In our proof below we shall follow \cite{NS74} by using a local finite element projection in order to split the finite element error into an approximation error and a ``discrete harmonic'' term which may be bounded using Lemma \ref{lem3-1}.  We note, however, that the use of a local finite element projection is not necessary, and our final local error estimate may in fact be proved with some simple modifications to the proof of Lemma \ref{lem3-1} above.  These two styles of proof are essentially equivalent.  Local finite element projections have been used for example in \cite{NS74}, \cite{SW77}, \cite{SW95}, and \cite{AL95} in order to prove local a priori error estimates.  The methodology of Lemma \ref{lem3-1} in which no local projections are used has been employed in for example \cite{De04b} and \cite{Guz06} in order to prove local a priori error estimates and in \cite{LN03} and \cite{Dem07} in order to prove local a posteriori error estimates.

\begin{theorem}
\label{t3-3}
Let $G_0 \subset G \subset \Omega_0$ be given, and let $dist(G_0, \partial G \setminus \partial \Omega)=d$ with $d \le \min\{2 d_0, d_1\}$ where $d_0$ is the parameter defined in the assumption R1 and $d_1$ is defined in Remark \ref{rem3-1}.  Let also A1, A2, and A3 hold with $\tilde{d}=\frac{d}{16}$.  In addition let $\max_{T \cap G \neq \emptyset} \frac{h_T}{d} \le \frac{1}{16}$.  Then
\begin{equation}
\begin{split}
\|u-u_h\|_{H^1(G_0)} \le & C \min_{u_h -\chi \in S_{D,0} } (\|u-\chi\|_{H^1(G)}+\frac{1}{d} \|u-\chi\|_{L_2(G)})
\\&+C \frac{1}{d} \|u-u_h\|_{L_2(G)}.
\label{eq3-12}
\end{split}
\end{equation}
Here $C$ depends only on the constant $C$ in (\ref{eq2-2}) and the coefficients of $L$.
\end{theorem}
\begin{proof}
We assume that $G_0$ is the intersection of a ball $B_{\frac{d}{2}}$ of radius $\frac{d}{2}$ with $\Omega_0$; the general case may be proved using a covering argument.  Let $G_1$ be the intersection with $\Omega_0$ of a ball having the same center as $G_0$ and having radius $\frac{3d}{4}$, and without loss of generality let $G$ be the corresponding ball of radius $d$.  Let then $\omega \in C_0^\infty(G)$ be a cutoff function which is $1$ on $G_1$ and which satisfies $\|\omega\|_{W_\infty^j(G)} \le Cd^{-j}$, $0 \le j \le r$.  Note that we may apply Lemma \ref{lem3-1} with $G_0$ on the left hand side of the estimate (\ref{eq3-9-b}) and $G_1$ on the right hand side.  

Next we let $P(\omega u)$ be a local finite element projection of $\omega u$.  In particular, we let $P(\omega u) \in S_D \cap H_<^1(G)$ with $u_h-P(\omega u) =0 $ on $\Gamma_D \cap \partial G_1$ satisfy 
\begin{equation}
L(\omega u-P(\omega u), v_h)=0, ~ v_h \in S_{D,0} \cap H_<^1(G).
\label{eq3-13}
\end{equation}
The local coercivity condition (\ref{eq3-2-1}) then implies the stability estimate
\begin{equation}
\|P(\omega u)\|_{H^1(G)} \le C \|\omega u \|_{H^1(G)}.
\label{eq3-14} 
\end{equation}

Recalling that $u_h-P(\omega u)=0$ on $\Gamma_D \cap \partial G_1$ while employing (\ref{eq3-9-b}) and using (\ref{eq3-2-1a}) while recalling that $\omega \equiv 1$ on $G_1$, we compute that
\begin{equation}
\begin{split}
\|u-u_h&\|_{H^1(G_0)}   \le \|\omega u- P(\omega u)\|_{H^1(G_0)}+\|P(\omega u)-u_h\|_{H^1(G_0)}
\\ \le  & \|\omega u-P(\omega u)\|_{H^1(G)}+\frac{C}{d}\|P(\omega u)-u_h\|_{L_2(G_1)}
\\ \le & \|\omega u-P(\omega u)\|_{H^1(G)}+\frac{C}{d}(\|P(\omega u)-\omega u\|_{L_2(G_1)}+\|u-u_h\|_{L_2(G_1)})
\\ \le & C \|\omega u-P(\omega u)\|_{H^1(G)}+\frac{C}{d} \|u-u_h\|_{L_2(G_1)}.
\label{eq3-16}
\end{split}
\end{equation}
Next we employing the triangle inequality along with (\ref{eq3-14}) while recalling that $\|\omega \|_{W_\infty^j(G_2)} \le Cd^{-j}$ in order to find that
\begin{equation}
\begin{split}
\|\omega u-P(\omega u)\|_{H^1(G)} \le & C \|\omega u\|_{H^1(G)}
\\ \le & C(\|u\|_{H^1(G)}+\frac{1}{d} \|u\|_{L_2(G)}).
\end{split}
\label{eq3-17}
\end{equation}

In order to complete the proof of (\ref{eq3-12}), we first insert (\ref{eq3-17}) into (\ref{eq3-16}) and finally write $u-u_h=(u-\chi)+(\chi-u_h)$ with $u_h-\chi \in S_{D,0}$.  

\end{proof}

\bibliographystyle{amsalpha}
\providecommand{\bysame}{\leavevmode\hbox to3em{\hrulefill}\thinspace}
\providecommand{\MR}{\relax\ifhmode\unskip\space\fi MR }
\providecommand{\MRhref}[2]{%
  \href{http://www.ams.org/mathscinet-getitem?mr=#1}{#2}
}
\providecommand{\href}[2]{#2}

\end{document}